\documentclass{amsart}

\usepackage{amsmath,amsthm,amscd,amssymb}
\usepackage{mathrsfs}
\usepackage{graphicx}
\usepackage{enumerate}

\usepackage{color}

\theoremstyle{plain}
\newtheorem{thm}{Theorem}

\newtheorem{lem}[thm]{Lemma}
\newtheorem{prop}[thm]{Proposition}

\newtheorem{problem}[thm]{Problem}
\theoremstyle{definition}
\newtheorem{dfn}[thm]{Definition}

\theoremstyle{remark}
\newtheorem{remark}[thm]{Remark}

\numberwithin{equation}{section}

\newcounter{tmp}

\newcommand{\Ci}{\mathscr{C}}
\newcommand{\Sone}{\mathbb{S}^1}
\newcommand{\N}{\mathbb{N}}
\newcommand{\R}{\mathbb{R}}
\newcommand{\Z}{\mathbb{Z}}

\newcommand{\esssup}{\operatorname*{ess\,sup}}

\title[Historic behaviour for nonautonomous contraction mappings]{Historic behaviour for nonautonomous contraction mappings}

\date{\today}

\author{Shin Kiriki}
\address[Shin Kiriki]{Department of Mathematics, Tokai University, 4-1-1 Kitakaname, Hiratuka, Kanagawa, 259-1292, JAPAN}
\email{kiriki@tokai-u.jp}

\author{Yushi Nakano}
\address[Yushi Nakano]{\color{black} Department of Mathematics, Tokai University, 4-1-1 Kitakaname, Hiratuka, Kanagawa, 259-1292, JAPAN\color{black}}
\email{\color{black} yushi.nakano@tsc.u-tokai.ac.jp\color{black}}

\author{Teruhiko Soma}
\address[Teruhiko Soma]{Department of Mathematical Sciences, Tokyo Metropolitan University, 1-1 Minami-Ohsawa, Hachioji, Tokyo, 192-0397, JAPAN}
\email{tsoma@tmu.ac.jp}



\begin{document}

\begin{abstract}
We consider a parametrised perturbation of a $\Ci^r$ diffeomorphism on a closed smooth Riemannian manifold  with $r\geq 1$, modeled by nonautonomous dynamical systems. A point without time averages for a (nonautonomous) dynamical system is said to have historic behaviour. It is known that for any $\Ci ^r$ diffeomorphism, the observability of historic behaviour, in the sense of the existence of a positive Lebesgue measure set consisting of points with historic behaviour, disappears under absolutely continuous, independent and identically distributed (i.i.d.) noise. 
On contrast, we show that the observability of historic behaviour can appear by a non-i.i.d.~noise: we consider a contraction mapping for which the set of points with historic behaviour is of zero Lebesgue measure and provide an absolutely continuous, non-i.i.d.~noise under which the set of points with historic behaviour is of positive Lebesgue measure.
\end{abstract}

\maketitle

\section{Introduction}
 This \color{black} paper concerns nonautonomous dynamical systems on a parametrised family of $\Ci ^r$ diffeomorphisms on a closed smooth Riemannian manifold $M$ with $r\geq 1$. 
 Given a mapping $\theta : \Omega \to \Omega$ on base set $\Omega$, a \emph{nonautonomous dynamical system} (abbreviated NDS henceforth) on $M$ over $\theta$ is given as a mapping $F:\mathbb N_0\times \Omega \times M\to M$ satisfying  $F(0,\omega ,\cdot )= \mathrm{id} _M$  for each $\omega \in \Omega$ and the cocycle property
\[
F(n+m , \omega ,x ) = F(n , \theta ^m \omega , F(m , \omega , x )) ,\quad n,m\in \mathbb N_0, \quad \omega \in \Omega ,\quad x\in M.
\]
Here $\theta \omega$ denotes the value $\theta (\omega )$, and $\theta $ is called  a \emph{driving system}. 
The notation of nonautonomous dynamical systems has emerged as an abstraction of 
  \emph{random dynamical systems} (see Remark  \ref{rmk:1} for the precise definition of random dynamical systems; a standard reference  is the monograph by Arnold  \cite{Arnold}, see also  \cite{BDV04, JKR2015} for representation of  Markov chains of random perturbations by random maps).
For general properties of NDS, we refer to Kloeden and Rasmussen \cite{KR2011}.
Here it is merely stated that if we denote $F(n,\omega , \cdot )$ and $F(1,\omega , \cdot )$ by $f^{(n) }_\omega $ and $f_\omega$, respectively,  then we have
\begin{equation}\label{eq:basic}
F(n,\omega ,\cdot )\equiv f^{(n) }_\omega = f_{\theta ^{n-1} \omega } \circ f_{\theta ^{n-2} \omega } \circ \cdots \circ f_\omega.
\end{equation}
Conversely, it is straightforward to see that given a mapping $f: \Omega \times M \to M$: $(\omega ,x)\mapsto f_\omega (x)$, 
a mapping $F: \mathbb N_0 \times \Omega \times M\to M$  defined by \eqref{eq:basic} 
is an NDS over $\theta$. We call it the \emph{NDS induced by $f$} over $\theta$.

A  naive expectation
from  \eqref{eq:basic} is that once we impose an appropriate condition on $\omega \mapsto f_\omega$, the statistical properties  of the driving system $\theta$ (with respect to a given probability measure $\mathbb P$ on $\Omega$) will be transmitted to those of  $\{ f^{(n)} _\omega \}_{n\geq 0}$ ($\mathbb P$-almost surely).
A celebrated result in the direction is established by Ara\'ujo \cite{Araujo2000} for historic behaviour  in the i.i.d.~case, 
 which inspires the work in this paper.
(For another result in the direction from the viewpoint of mixing property or limit theorems, refer to \cite{MK2006, NTW2016, AA2017} and the references therein.) 
 To state his and our result, we  define historic  behaviour for $F$.

 \begin{dfn}\label{dfn:hb}
For given $\omega \in \Omega$, we say that the forward orbit of $x\in M$  at $\omega$  
has \emph{historic  behaviour} if there exists a continuous function $\varphi :M\to \R$ for which the time average
\begin{equation}\label{eq:hb1}
\lim _{n\to \infty}\frac{1}{n} \sum _{j=0} ^{n-1} \varphi (f^{(j)} _{\omega} (x))
\end{equation}
does not exist. 
For short, we call $x$ a point with historic behaviour at $\omega$.
\end{dfn}

The  concept of historic behaviour was  introduced by Ruelle \cite{Ruelle2001} for autonomous dynamical systems: Let $f_0 : M\to M$ be a $\Ci ^r$ diffeomorphism on $M$ and  $f_0^n$ the usual $n$-th iteration of $f_0$ with $n\geq 0$. Then, $\mathbb N_0\times M\ni (n, x) \mapsto f_0^n(x)$ is an (autonomous) dynamical system, and 
a point  $x\in M$  is said to have
\emph{historic  behaviour} if there exists a continuous function $\varphi :M\to \R$ for which the time average
$
\lim _{n\to \infty}\frac{1}{n} \sum _{j=0} ^{n-1} \varphi (f_0^j (x))
$
does not exist.
Since several statistical quantities are given as the time average of some function $\varphi$, 
 it is   natural to investigate the observability of historic behaviour in the sense of the existence of  a positive Lebesgue measure set consisting of point with  historic behaviour. \color{black}
In the autonomous situation, Bowen's famous folklore example  \cite{Takens1994} tells 
that there is a $\Ci ^\infty$ diffeomorphism on a  compact surface for which the set of points with historic behaviour is of positive Lebesgue measure. 
However, his example was not stable under small perturbations.
Hence, Takens asked in \cite{Takens2008} 
 whether there is a persistent class of $\Ci ^r$ diffeomorphisms for which the set of points with historic behaviour is of positive Lebesgue measure (called \emph{Takens' Last Problem}):
Very recently it was affirmatively answered  by the first and third authors  in \cite{KS2017}, that will be briefly restated (in a slightly stronger form) in Theorem \ref{thm:Newhouse}.
Furthermore, this was applied to detect a persistent class of 3-dimensional flows having  a positive Lebesgue measure set consisting of points with historic behaviour  in \cite{LR2016}. 
The reader is asked to see   \cite{Takens1994, Ruelle2001, Takens2008, BDV04} for the background of historic  behaviour in the autonomous situation.
 
In the nonautonomous situation, the first result about historic behaviour  was obtained by   Ara\'ujo   for parametrised perturbations of $\Ci ^r$ diffeomorphisms under i.i.d.~noise:
  if a parametrised perturbation of a $\Ci ^r$ diffeomorphism $f_0$ given as an i.i.d.~NDS $F$ is \emph{absolutely continuous}, then the set of  points with historic behaviour is a zero Lebesgue measure set 
(see Appendix \ref{appendix:araujo}  for the definition of absolute continuous i.i.d.~NDS's and the precise statement of Ara\'ujo's theorem). 
We can choose  the unperturbed system $f_0$ as a $\Ci ^r$ diffeomorphism for which the set of points with historic behaviour is persistently of positive Lebesgue measure, that means the disappearance of  historic behaviour under i.i.d.~noise (although the existence of a residual set consisting of points with historic behaviour for expanding maps is preserved under any random perturbations, refer to  \cite{Nakano2017}).   
Our purpose in this paper is to show the appearance of historic behaviour under some ``historic'' noise.

 \subsection{Setting and result}\label{subsection:setting}
Let $M$ be the circle given by $M =\R /\Z$. We endow $M$ with a metric $d_{M}(\cdot ,\cdot)$, where $d_{M}(x, y)$ is the infimum of  $\vert \tilde x-\tilde  y \vert$ over all representatives $\tilde x,\tilde y$ of $x,y \in M$, respectively.
Let $\pi _{M}:\R \to M$ be the canonical projection on the circle, i.e., $\pi _{M} (\tilde x)$ is the equivalent class of $\tilde x\in \R$.  
We write $I_0$ for $\pi _{M}\left([\frac{1}{4}, \frac{3}{4}]\right)$. Let $f_0$ be a $\Ci^r$ diffeomorphism on $M$ such that 
\begin{equation}\label{eq:n1}
f_0(x) =\frac{1}{2} x+\frac{1}{4} \mod 1,\quad x\in I_0,
\end{equation}
and that $\inf _{x\in M} Df_0(x) > 0$ (see Figure \ref{Fig_1_2}). 
We also assume that $f_0$ has exactly 
one source. 
 Then, it is not difficult to see that the set of points with historic behaviour for $f_0$ is an empty set, in particular, a zero Lebesgue measure set. (Note  that basin of attraction of $\pi _M(\frac{1}{2})$ is the whole space $M$ except the source.)

Next we introduce our main hypothesis for driving systems. 
Let $\Omega$ be a metric space equipped with the Borel $\sigma$-field, and $\mathbb P$ a probability measure on $\Omega$. Let $\theta  : \Omega  \to  \Omega$ be a continuous mapping.
Given an integer $\nu \geq 0$,
  $\omega \in \Omega$ 
   and an open set $U \subset \Omega$, 
   we say that an integer $j $ is in a \emph{$\nu$-trapped period} of 
    $\omega$ for  $U$ 
if   $j \geq \nu$  and $\theta ^{j-i} \omega \in U $ 
 for all  $i\in [0, \nu ]$.
For  
$n\geq  1$, we set
\begin{equation*}
N  _{\nu }( \omega , U; n) = \# \{ j \in [0 ,n-1] \; : \; \text{$j$ is in a $\nu$-trapped period of $\omega $ for $U$} \} .
\end{equation*}
Let $U_\delta (\omega )$ be the ball of $\omega \in \Omega$ with radius $\delta >0$.
We will assume the following condition:
\begin{itemize}
\item[(H)] There is a $\mathbb P$-positive measure set $\Gamma \subset \Omega$ and distinct points $p$ and $\hat p$  such that for any integer $\nu \geq 0$ and positive number $\delta$, one can find two distinct real numbers $\lambda _1$ and $\lambda _2$ in $[0, 1]$ and subsequences $\{ n_1(J) \} _{J\geq 1}$ and $\{ n_2 (J) \} _{J\geq 1}$ of $\mathbb N$ such that
 \begin{equation*}
\lim _{J\to \infty} \frac{N_{\nu } ( \omega , U_\delta (p) ; n _i(J)) }{n _i (J)} =1 -\lambda _i ,  \quad  \lim _{J\to \infty}  \frac{N_{\nu } (\omega , U_\delta (\hat p) ;   n _i (J)) }{n _i(J)} = \lambda _i
 \end{equation*}
 for  $i=1, 2$. 
\end{itemize}

Let $\kappa :\Omega \to [-1,1]$ be a surjective continuous function such that $\kappa (p) \neq \kappa (\hat p)$ and that the pushforward $\kappa _* \mathbb P$ of $\mathbb P$ by $\kappa$ is absolutely continuous with respect to $\mathrm{Leb} _{\mathbb R}$. 
Fix  a noise level $0<  \epsilon <\frac{1}{8}$. 
We define  a parametrised perturbation $f: \Omega \times M\to M$  of $f_0$ by
\begin{equation}\label{eq:defoffepsilon}
 f_\omega (x) \equiv
  f(\omega , x) = f_0(x)+ \epsilon \kappa (\omega ) \mod 1, \quad (\omega , x) \in \Omega \times M .
\end{equation}




Now we can provide our main theorems:
\begingroup
\setcounter{tmp}{\value{thm}}
\setcounter{thm}{0}
\renewcommand\thethm{\Alph{thm}}
\begin{thm}\label{thm:main}
Suppose that $\theta$ satisfies the condition (H).
Let $F$ be the NDS induced by $f$ in 
 \eqref{eq:defoffepsilon} over $\theta$. 
Then  for any $\omega \in \Gamma $, there exists a positive Lebesgue measure set (including $I_0$) consisting of  points with historic behaviour at $\omega$.
\end{thm}


For an application of Theorem \ref{thm:main}, 
we will show that the condition (H) can be satisfied by the classical Bowen example. 
Let $\Omega $ be a  compact surface and $\mathbb P$ the normalised Lebesgue measure of $\Omega$. 
\begin{thm}\label{thm:Bowen}
The time-one map $\theta$ of the Bowen flow (definition  given in Subsection \ref{subsection:Bowen}) on $\Omega$ satisfies the condition (H). 
\end{thm}

We will also show that the persistent driving systems in \cite{KS2017} satisfy the condition (H).
Let $\mathrm{Diff}^{\tilde r} (\Omega ,\Omega )$ be the set of all $\mathscr C^{\tilde r}$ diffeomorphisms on $\Omega$ endowed with the usual $\mathscr C ^{\tilde r}$ metric with $2 \leq  \tilde r < \infty$, and let $\mathcal N \subset \mathrm{Diff}^{\tilde r}(\Omega , \Omega )$ be a Newhouse open set.\footnote{
For each $\tilde \theta \in \mathrm{Diff} ^{\tilde r}(\Omega ,\Omega)$ with a saddle fixed point $\tilde p$ with $\tilde r\geq 2$, one can find an open set  $\mathcal{N}$ in $\mathrm{Diff} ^{\tilde r}(\Omega ,\Omega)$ (called  a \emph{Newhouse open set})
such that the closure of $\mathcal N$ contains $\tilde \theta$ and 
any element of $\mathcal{N}$  is arbitrarily $\mathscr C ^{\tilde r}$-approximated by a diffeomorphism 
$\theta $ with 
a homoclinic tangency associated with 
the continuation $p$ of $\tilde p$, and moreover $\theta $ has a $\mathscr C ^{\tilde r}$-persistent tangency associated with some  
nontrivial hyperbolic set $\Lambda $ containing $p$ 
(i.e.~there is a $\mathscr C ^{\tilde r}$ neighborhood of $\theta $ any element of which has a homoclinic tangency 
for the continuation of $\Lambda$). See \cite{Newhouse79}.
}

\begin{thm}\label{thm:Newhouse}
 There exists a dense subset $\mathcal D$ of $\mathcal N$ such that all  $\theta \in \mathcal D$ satisfies the condition (H). 
\end{thm}

\endgroup

\setcounter{thm}{\thetmp}

\subsection{Problem}
Before starting the proofs of main theorems, we briefly  consider historic behaviour for nonautonomous contraction mappings in more general setting.
Let $f_0$ be as in \eqref{eq:n1}. 
Let $(\Omega , \mathcal F, \mathbb P)$ be a probability space.
Let $\kappa$ be a measurable  function on $\Omega$ with values in $[-1, 1]$ $\mathbb P$-almost surely.
We define  $f: \Omega \times M\to M$   by
\begin{equation}\label{eq:defoffepsilon2}
 f_\omega (x) \equiv f(\omega , x) = f_0(x)+ \epsilon \kappa (\omega ) , \quad (\omega , x) \in \Omega \times M .
\end{equation}
Furthermore, we assume that $\theta$ is nonsingular with respect to $\mathbb P$ (i.e.,~$\mathbb P(\theta ^{-1} \Gamma )=1$ if $\Gamma$ is measurable and $\mathbb P(\Gamma ) =1$). 

We say that the driving system $\theta$ is 
\emph{historic}
if there exists a positive measure set $\Gamma$ with respect to $\mathbb P$ such that for each $\omega \in \Gamma$, one can find an integrable function $b:\Omega \to \mathbb R$ whose time average $\lim _{n\to \infty} \frac{1}{n} \sum _{j=0}^{n-1} b (\theta ^j \omega )$ does not exist. 
 Otherwise, we say that $\theta$ is \emph{non-historic}.

 \begin{remark}\label{rmk:1}
A measurable NDS $F$  over a measurable driving system $\theta$ 
is said to be a \emph{random dynamical system} (abbreviated RDS) if $\theta$ is measure-preserving (refer to  \cite{Arnold};  important examples of random dynamical systems are i.i.d.~NDS's, see Appendix \ref{appendix:araujo}). 
It follows from Birkhoff's ergodic theorem that any measure-preserving driving system is non-historic. 
That is,  any random dynamical system is an NDS over a non-historic driving system.
See Figure \ref{fig-0}.
\end{remark}

\begin{figure}[hbtp]
\centering
\scalebox{0.7}{\includegraphics[clip]{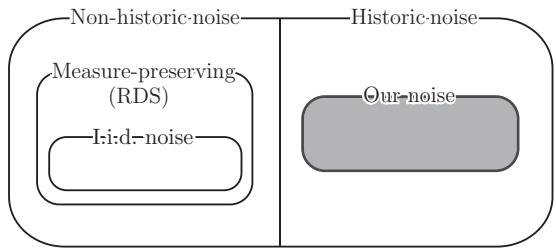}}
\caption{Classification of nonautonomous dynamical systems.}
\label{fig-0}
\end{figure}

 The following proposition can be shown by  a standard graph-transformation argument, but might be suggestive for historic behaviour of nonautonomous contraction mappings.  See Appendix \ref{appendix:proof} for the proof.
\begin{prop}\label{prop:4}
Let  $\theta$  be measurably invertible.
Let $f$ be as in \eqref{eq:defoffepsilon2} and 
 $F$  the NDS induced by $f$ over a driving system $\theta$.
 Suppose that  $\theta$ is non-historic.
Then for $\mathbb P$-almost every $\omega$, the set of points in $I_0$ with historic behaviour at $\omega$ is an empty set, in particular, a Lebesgue zero measure set.
\end{prop}

Comparing  Theorem \ref{thm:main} with proposition \ref{prop:4}, one may naturally ask the following problem.
\begin{problem}\label{prob:2}
Let $f$ be as in \eqref{eq:defoffepsilon2} and 
 $F$  the NDS induced by $f$ over a driving system $\theta$.
 Suppose that  $\theta$ is historic.
Then under some mild condition on $\kappa$, can one find a positive measure set $\Gamma$ with respect to $\mathbb P$ such that  there exists a positive Lebesgue measure set (including $I_0$) consisting of  points with historic behaviour at $\omega$  for any $\omega \in \Gamma$?
\end{problem}

\begin{remark}
Apart from driving systems, one can consider generalisations of Theorem \ref{thm:main} to other unperturbed  systems $f_0$: 
It is highly likely that the existence of a positive Lebesgue measure set consisting of points with historic behaviour 
remains true  for any $\Ci ^r$ diffeomorphism on any closed smooth Riemannian manifold $M$, only by requiring that $f_0$ 
has a sink (with an appropriate modification on  the formulation of small perturbation $f$ in higher dimension; see Example 2 in \cite{Araujo2000}).
It might also be possible (and of great interest) to explore generalisation  to hyperbolic mappings $f_0$ by considering their transfer operators, refer to \cite{BKS}.
However,  in order to keep our presentation as transparent as possible, we restricted ourselves to  
the concrete example given in \eqref{eq:n1}.
\end{remark}

\section{Proofs}\label{section:proof}
\subsection{Proof of Theorem \ref{thm:main}}

 We start the proof of Theorem \ref{thm:main} by noting that $f _\omega (I_0) \subset I_0$  and $f_\omega |_{I_0} :I_0\to I_0$ has a unique fixed point, denoted by $X_\omega$,  for each $\omega \in \Omega$. 
In particular, for $\omega =p,\hat p$, 
\begin{equation}\label{eqn_XpXp}
X_p=\frac12+2\epsilon \kappa(p),\quad
X_{\hat p}=\frac12+2\epsilon \kappa(\hat p) .
\end{equation}
See Figure \ref{Fig_1_2}.
\begin{figure}[hbtp]
\centering
\scalebox{0.6}{\includegraphics[clip]{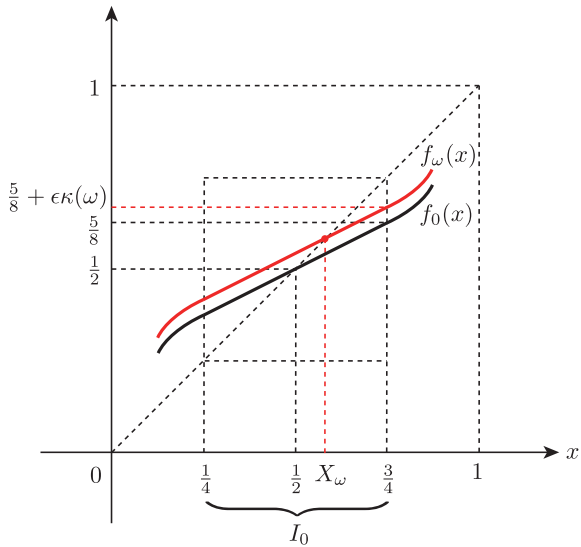}}
\caption{Locally contraction mappings $f_0$ and $f_\omega$.}
\label{Fig_1_2}
\end{figure}


We need the following elementary lemma. 
\begin{lem}\label{lem:1}
For any   $n\in \N_0$, $x\in I_0$ and $\omega ,\omega ^\prime \in \Omega$, we have
\begin{equation}\label{eq:lem1}
d_{M}\left( f^{(n)} _{\omega}(x) ,X_{\omega ^\prime } \right) \leq \frac{1}{2^n} +6\epsilon \max_{0\leq j\leq n-1} \left\vert \kappa(\theta ^j \omega )  - \kappa(\omega ^\prime )\right\vert .
\end{equation}
\end{lem}
\begin{proof}
Fix  $n\in \N_0$, $x\in I_0$ and $\omega ,\omega ^\prime \in \Omega$. 
Noting that $X_{\omega ^\prime }=\pi _{M}(\frac{1}{2}+2\epsilon \kappa(\omega ^\prime))$ together with \eqref{eq:basic}, we have
\begin{align*}
d_{M}\left(f^{(n)}_{\omega } (x)  ,X _{\omega ^\prime }\right)
&\leq d_{M}\left( f_{\theta ^{n-1} \omega }( f^{(n-1)}_{\omega } (x) ) , f_{\theta ^{n-1}\omega } (X_{\theta ^{n-1} \omega })\right) + d_{M}\left( X_{\theta ^{n-1}\omega }, X_{\omega ^\prime}\right)\\
&\leq \frac{1}{2}d_{M}\left( f^{(n-1)}_{ \omega } (x)  , X _{\theta ^{n-1} \omega }\right) +2\epsilon \left\vert \kappa(\theta ^{n-1}\omega )-\kappa(\omega ^\prime)\right\vert .
\end{align*}
Reiterating this argument, we finally get that $d_{M}\left( f^{(n)} _{\omega } (x)   ,X _{\omega ^\prime}\right)$ is bounded by
\begin{align*}
\frac{1}{2^n}d_{M}\left( x  , X _{\omega }\right) +2\epsilon \left(\sum _{j=1}^{n-1}\frac{\vert \kappa(\theta ^{n-j-1} \omega )-\kappa(\theta ^{n-j}\omega )\vert }{2^j} +\left\vert \kappa(\theta ^{n-1}\omega )-\kappa (\omega ^\prime ) \right\vert \right).
\end{align*}
Hence,  the conclusion follows from the triangle inequality
$$\vert \kappa(\theta ^{n-j-1}\omega )-\kappa(\theta ^{n-j}\omega )\vert\leq \vert \kappa(\theta ^{n-j-1}\omega )-\kappa(\omega ')\vert
+\vert \kappa(\omega ')-\kappa(\theta ^{n-j}\omega )\vert.$$
This completes the proof.
\end{proof}

We continue the proof of Theorem \ref{thm:main}.
We let $V(X_p)$ and  $V(X_{\hat p})$ be the $\rho _0$-neighbourhoods of 
$X_p$, $X_{\hat p}$ in $M$, respectively, with $\rho _0 =\frac{\epsilon\vert \kappa(p) -\kappa (\hat p) \vert }{2}$. 
By \eqref{eqn_XpXp}, $V(X_p)\cap V(X_{\hat p})=\emptyset$.
Fix a positive integer $\nu _0$ satisfying
\begin{equation}\label{eq:e0324a}
\frac{1}{2^{\nu _0} }\leq \frac{\rho _0}{3}.
\end{equation}
Furthermore, we let $\delta _0^\prime $ be a positive number such that $d_\Omega (\omega ,p_1 )<\delta _0^\prime $ implies $\vert \kappa (\omega )-\kappa (p_1) \vert \leq  \frac{\rho _0 }{18} $ with $p_1 =p$ and $\hat p$, and set
\begin{equation}\label{eqn_delta}
\delta _0=\min\left\{\delta _0^\prime ,\frac{\mathrm{dist}_\Omega (p,\partial \Omega)}2, \frac{\mathrm{dist}_\Omega (\hat p,\partial \Omega)}2\right\} ,
\end{equation}
so that $U_{\delta _0}(p)\cap U_{\delta _0}(\hat p)=\emptyset$ 
and  $(U_{\delta _0}(p)\cup U_{\delta _0}(\hat p))\cap \partial \Omega=\emptyset$.

Let $j$ be in a  $\nu _0$-trapped period of $\omega \in \Omega $ for $U_{\delta _0} ( p )$. 
Then, 
we have
 \begin{equation}\label{eqn_dSfX2}
  \max _{0\leq i \leq \nu _0 } d_\Omega (\theta ^{j -i}\omega , p )\leq \delta _0 .
\end{equation}
 On the other hand,  applying \eqref{eq:lem1} with $n ,x, \omega$  and  $\omega ^\prime$ replaced by $\nu _0$, $f^{(j-\nu _0)}_{\omega }(x)$, $\theta ^{j-\nu _0}\omega$ and $p$ together with \eqref{eq:basic}, we have
\[
d_{M}(f_{\omega }^{(j)} (x), X_{p}) \leq \frac{1}{2^{\nu _0}} +6\epsilon \max _{0\leq i \leq \nu _0 } \left\vert \kappa(\theta ^{j -i}\omega ) -\kappa (p )\right\vert 
\]
for all $x\in I_0$.
Therefore, it follows from \eqref{eq:e0324a}, \eqref{eqn_delta}
 and \eqref{eqn_dSfX2} that 
\[
d_{M}(f_{\omega }^{(j)} (x), X_{p}) \leq \frac{2}{3} \rho _0 ,
\]
that is, $f_{\omega }^{(j)} (x) \in V(X_p)$. 
A similar argument implies that  if $j$ is in a  $\nu _0$-trapped period of $\omega \in \Omega $ for $U_{\delta _0} (\hat p)$, then $f_{\omega }^{(j)} (x) \in V(X_{\hat p})$ for any $x\in I_0$.

We assume that $\lambda _1<\lambda _2$ without loss of generality. 
Let $\varphi _0 :M \to [0,1] \subset \R $ be a nonnegative-valued continuous function  such that $\varphi _0(x) =1$ if  $x$ is in $V(X_p)$ and $\varphi _0(x)=0$ if $x$ is in $V(X_{\hat p})$. 
For each  $x\in I_0$ and $\omega \in \Gamma$, by the condition (H)  together with observation in the previous paragraph, we have
\[
\frac{1}{n_1(J)} \sum _{j=0}^{n_1(J)} \varphi _0 (f_\omega ^{(j)}(x)) 
\geq \frac{ \# \{ 0\leq j\leq n_1(J) \mid f^{(j)}_\omega (x) \in V(X_p)\} }{n_1(J)} \to 1- \lambda _1
\]
and
\[
\frac{1}{n_2(J)} \sum _{j=0}^{n_2(J)} \varphi _0 (f_\omega ^{(j)}(x)) 
\leq 1- \frac{ \# \{ 0\leq j\leq n_2(J) \mid f^{(j)}_\omega (x) \in V(X_{\hat p})\} }{n_2(J)} \to 1- \lambda _2
\]
as $J\to \infty$.
Therefore, we get
\[
 \liminf _{n\to \infty}\frac{1}{n} \sum ^{n}_{j=0} \varphi  _0(f_{\omega } ^{(j)} (x) )\leq 1-\lambda _2 <1 - \lambda _1 \leq \limsup _{n\to \infty}\frac{1}{n} \sum ^{n}_{j=0} \varphi _0(f_{\omega } ^{(j)} (x) )
\]
for all $(\omega ,x)$ in $\Gamma \times I_0$. 
This completes the proof of Theorem \ref{thm:main}.

\subsection{Proof of Theorem \ref{thm:Bowen}}\label{subsection:Bowen}

It is mentioned in \cite{Takens1994} that Bowen considered a surface flow $\{ \theta ^t \} _{t\in \mathbb R}$ 
 generated by a smooth (at least $\mathscr C ^3$) vector field with two saddle points $p$ and $\hat p$ and two heteroclinic orbits $\gamma $ and $\hat \gamma $ connecting the points, which are included in the unstable and stable manifolds of $p$ respectively, such that the closed curve $\gamma := p \sqcup \hat p \sqcup \gamma  \sqcup \hat \gamma $ is 
  attracting in the following sense: if we denote  the expanding and contracting eigenvalues of the linearised vector field around $p$ by $\alpha _+$ and $-\alpha _-$, and  the ones around $\hat p$ by $\beta _+$ and $- \beta _-$, then
 $\alpha _- \beta _- > \alpha _+ \beta _+$. 
Let $\Gamma$ be the intersection of the  basin of attraction of $\gamma$ and the open set surrounded by  $\gamma$, which is a nonempty open set of $\Omega$. 
Furthermore, we take sections $\Sigma $ and $\hat \Sigma $ transversally  intersecting $\gamma$ and $\hat \gamma $, respectively.
See Figure  \ref{Bowen-eye} for configuration.

\begin{figure}[hbtp]
\centering
\scalebox{0.6}{\includegraphics[clip]{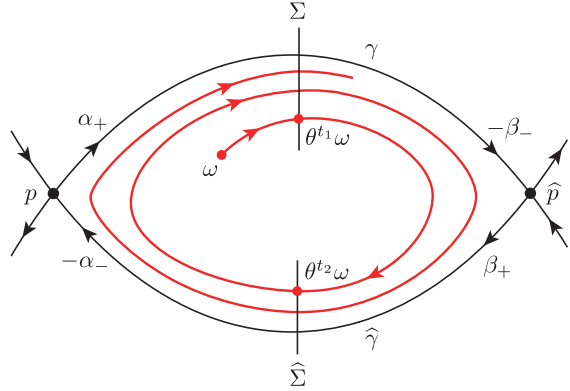}}
\caption{The Bowen flow.}
\label{Bowen-eye}
\end{figure}

Fix $\omega  \in \Gamma$. 
Let $\{ t_j \} _{j\geq 1}$ be  successive times at which the forward orbit of $\omega $ by $\{ \theta ^t\} _{t\in \mathbb R}$ intersects $\Sigma $ and $\hat \Sigma $. 
By taking the sections smaller, one can assume that $\theta^{t_{j}}\omega  \in \Sigma $ if $j$ is odd and $\theta^{t_{j}}\omega \in \hat \Sigma $ if $j$ is even. 
Let $T^{(p)}_j= t _{2j+1}- t_{2j}$ and $T^{(\hat p)}_j= t_{2j}- t_{2j-1}$.
It was shown in 
 \cite{Takens1994} that 
\begin{equation}\label{eq:rrrr1}
\lim _{j\to \infty } \frac{T^{(p)}_{j} }{T^{(\hat p)} _{j}} = \sigma _1, \quad \lim _{j\to \infty } \frac{T^{(\hat p)} _{j+1}}{T^{( p)} _{j}} = \sigma _2
\end{equation}
with $\sigma _1= \frac{\beta _-}{\alpha _+}$ and $\sigma _2= \frac{\alpha _-}{\beta _+}$,
 and 
%
 \begin{equation}\label{eq:rrrr1b}
 \lim _{j\to \infty} \frac{T^{(p)}_{\delta ,j}}{T^{(p)}_j} = \lim _{j\to \infty} \frac{T^{(\hat p)}_{\delta ,j} }{T^{(\hat p)}_j} =1
 \end{equation}
 for each $\delta  > 0$, where   
  $T^{(p)}_{\delta ,j}$ and $ T^{(\hat p)}_{\delta ,j}$ are the lengths of  $ \{ t_{2j} \leq t\leq t_{2j+1} \mid \theta ^t \omega \in U_\delta (p) \}$ and $ \{ t_{2j-1} \leq t\leq t_{2j} \mid \theta ^t \omega \in U_\delta (\hat p) \}$, respectively,  
when the lengths are well-defined (in particular, for each sufficiently large $j$).
%

 
 Let $n_1(J) = [\hat T_{2J-1}]$ and $n_2(J) = [\hat T_{2J}]$ with  the notation $[t]$ for the integer 
part of $t\in \mathbb R$. 
 Given $\delta >0$ and $\nu \geq 0$, let $J_0$ be an integer such that $\min \{T_{\delta , J_0}^{(p)} , T_{\delta , J_0} ^{(\hat p)} \} \geq \nu +2$.
Then, for any $J\geq J^\prime \geq J_0$, we have
\begin{align*}
\frac{N_\nu (\omega ,U_\delta (p) ; n_1(J)  )}{ n_1(J)   } &\geq \frac{\sum _{j=J^\prime } ^J (T_{\delta ,j} ^{(p)} -2 - \nu )}{ \sum _{j=1} ^J (T_j^{(p)} + T_j^{(\hat p)}  ) } \\
&=(1-Z_1) \cdot \frac{\sum _{j=J^\prime } ^J T_{\delta ,j} ^{(p)} }{ \sum _{j=J^\prime } ^J (T_j^{(p)} + T_j^{(\hat p)}  ) }  -Z_2 ,
\end{align*}
where $Z_1= \sum _{j=1}^{J^\prime -1} (T_j^{(p)} +T^{(\hat p)} _j)  / \sum _{j=1}^{J} (T_j^{(p)} +T^{(\hat p)} _j)  $ and $Z_2 =(2+\nu) /\sum _{j=1}^{J} (T_j^{(p)} +T^{(\hat p)} _j)$, both of which go to $0$ as $J\to \infty$ for any fixed $J^\prime$.  
Hence, by \eqref{eq:rrrr1} and \eqref{eq:rrrr1b}, it is straightforward to see that for any $\tilde \epsilon >0$, there is an integer $J_{\tilde \epsilon } \geq J_0$ such that for each $J\geq J_{\tilde \epsilon}$,
\[
\frac{N_\nu (\omega ,U_\delta (p) ; n_1(J)  )}{ n_1(J)   } \geq (1-\tilde \epsilon )\cdot \frac{(\sigma _1-\tilde \epsilon) T_{J} ^{(\hat p)}\sum _{j=0} ^{J-J_{\tilde \epsilon} } \{ (\sigma _1 -\tilde \epsilon )(\sigma _2 -\tilde \epsilon ) \} ^{-j}}{(1+(\sigma _1+\tilde \epsilon ) ) T_{J} ^{(\hat p)}\sum _{j=0} ^{J-J_{\tilde \epsilon} } \{ (\sigma _1 +\tilde \epsilon )(\sigma _2 +\tilde \epsilon ) \} ^{-j}} -\tilde \epsilon .
\]
Since $\tilde \epsilon$ is arbitrary, we have
\[
\liminf _{J\to \infty} \frac{N_\nu (\omega ,U_\delta (p) ; n_1(J)  )}{ n_1(J)   } \geq \frac{\sigma _1}{1+\sigma _1} .
\]


In a similar manner we can show that
\[
\liminf _{J\to \infty} \frac{N_\nu (\omega ,U_\delta (\hat p) ; n_1(J)  )}{ n_1(J)   } \geq \frac{1}{1+\sigma _1}.
\]
and that
\[
\liminf _{J\to \infty} \frac{N_\nu (\omega ,U_\delta (p) ; n_2(J)  )}{ n_2(J)   } \geq \frac{1}{1+\sigma _2} , \quad \liminf _{J\to \infty} \frac{N_\nu (\omega ,U_\delta (\hat p) ; n_2(J)  )}{ n_2(J)   } \geq \frac{\sigma _2}{1+\sigma _2} .
\]
This completes the proof of Theorem \ref{thm:Bowen} with $\lambda _1 = \frac{1}{1+\sigma _1}$ and $\lambda _2 = \frac{\sigma _2}{1+\sigma _2}$.

%
%
%

\subsection{Proof of Theorem \ref{thm:Newhouse}}

In \cite{KS2017}, we have actually shown that, for sufficiently large positive integers $z_0$, $n_0$, $k_0$, 
there exists an element $\theta=\theta_{\boldsymbol{z}}$ in any small neighbourhood  of any $\Ci ^{\tilde r}$ diffeomorphism in the Newhouse open set $\mathcal N$ 
associated with any sequence $\boldsymbol{z}
=\{z_k\}_{k=k_0}^\infty$ of integers each entry of which is either $z_0$ or $z_0+1$ 
and there exists a sequence $\{R_k\}_{k=k_0}^\infty$ of mutually disjoint rectangles in $\Omega$ with $\mathrm{Int}R_{k_0}=\Gamma$ 
and satisfying the following conditions.
\begin{enumerate}
\renewcommand{\theenumi}{C\arabic{enumi}}
\item\label{R1}
$\lim_{k\to\infty}\mathrm{diam}\,(R_k)=0$.
\item\label{R2}
There are sequences $\{a_k\}_{k=k_0}^\infty$, $\{b_k\}_{k=k_0}^\infty$ of positive integers 
with
\begin{equation*}
\limsup_{k\to\infty}\dfrac{a_k}{k}<\infty\quad\text{and}\quad\limsup_{k\to\infty}\dfrac{b_k}{k}<\infty
\end{equation*}
and such that, for any $\delta >0$ and \color{black} $\omega \color{black} \in R_k$ with sufficiently large $k\geq k_0$, 
\begin{itemize}
\item
$\theta^{a_k+n_0+j}\color{black} \omega \color{black} \in U_\delta(p)$ if $j\in \{0,\dots,z_kk^2-2n_0\}$,
\item
$\theta^{a_k+n_0+z_kk^2+j}\color{black} \omega \color{black} \in U_\delta(\hat p)$ if $j\in \{0,\dots,k^2-2n_0\}$, 
\item
$\theta^{m_k}\color{black} \omega \color{black} \in \mathrm{Int}R_{k+1}$ 
for $m_k=(z_k+1)k^2+a_k+b_k$.
\end{itemize}
\end{enumerate}
\color{black} Furthermore, $n_0$, $k_0$, $\{a_k\}_{k=k_0}^\infty$ and $\{b_k\}_{k=k_0}^\infty$ can be taken independently of $\boldsymbol{z}
=\{z_k\}_{k=k_0}^\infty$. \color{black}
See Figure \ref{Fig_3} for the situation.
\begin{figure}[hbtp]
\centering
\scalebox{0.6}{\includegraphics[clip]{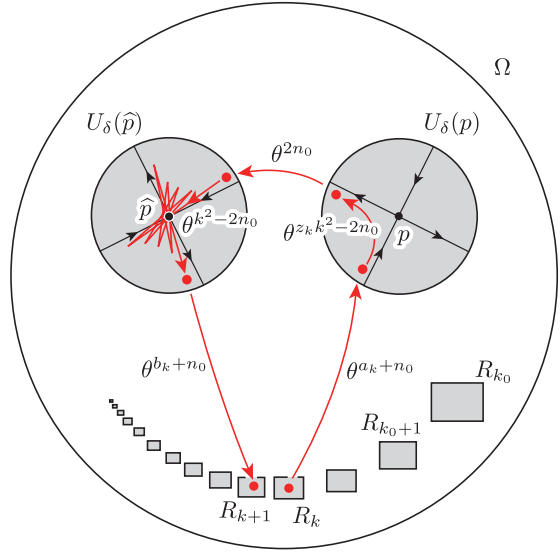}}
\caption{Travel from $R_k$ to $\mathrm{Int}R_{k+1}$ via $\theta^{m_k}$.
The case where the eigenvalues of $D\theta(p)$ are positive and those of $D\theta(\hat p)$ are negative.}
\label{Fig_3}
\end{figure}

 For a given monotone increasing sequence $\{ k(J^\prime ) \}_{J^\prime =1} ^\infty$ of integers with $k(1) >k_0$, 
the sequence $\boldsymbol{z}=\{ z_k\} _{k=k_0}^{\infty}$ is constructed to   satisfy 
\begin{equation*}
z_k=z_0\quad\text{if $J'$ is odd\quad and}\quad z_k=z_0+1\quad\text{if $J'$ is even}
\end{equation*}
for any $k(J'-1)< k\leq k (J')$.

Now we will show that the sequence $\{ k(J^\prime )\}_{J^\prime =1}^\infty$ can be taken so that 
the following inequality holds:  For any $\nu \geq 0$ and $\delta >0$, there is an integer  $J^\prime _0\geq 1$ such that if $J^\prime \geq J^\prime _0$, then 
\begin{equation}\label{eq:e0326c}
\frac{N_{\nu  } (\omega , U_\delta (p) ; \hat m_{k(J')})}{\hat m_{k(J')}} \geq \frac{z_*}{z_*+1} -2^{-J'},
\end{equation}
where $z_*=z_0$ if $J^\prime $ is odd and $z_*=z_0+1$ if $J^\prime $ is even. 
It follows from (C2) 
 that 
\begin{align*}
\label{eq:e0326b} 
\frac{N_{\nu  } (\omega , U_\delta (p) ; \hat m_{k(J')})}{\hat m_{k(J')}} &\geq \frac{\sum _{k= k(J'-1)+1}^{k(J')}  (z_k k^2 -2n_0 -\nu ) }{\hat m_{k(J')}}\\
&=  \frac{\sum _{k= k(J'-1)+1}^{k(J')}   (z_k k^2 -2n_0 -\nu ) }{\sum _{k= k(J'-1)+1}^{k(J')}   ((z_k+1)k^2 +a_k +b_k)} \left(1 - \frac{\hat m_{k (J' -1)}}{\hat m_{k(J')}}\right),
\end{align*}
for each $\nu \geq 0$, $\delta >0$ and sufficiently large $J^\prime \geq 1$.
On the other hand, it is easy to check that
\[
\frac{\sum _{k= k(J'-1)+1}^{k(J')}   (z_k k^2 -2n_0 -\nu ) }{\sum _{k= k(J'-1)+1}^{k(J')}   ((z_k+1)k^2 +a_k +b_k)} 
= \frac{z_* -Z_1}{z_*+1+Z_2},
\]
where $Z_1=\sum _{k=k(J'-1)+1}^{k(J')} (2n_0+\nu) / \sum _{k=k(J'-1)+1}^{k(J')} k^2$ and $Z_2 =\sum _{k=k(J'-1)+1}^{k(J')} (a_k + b_k )/\sum _{k=k(J'-1)+1}^{k(J')} k^2$. 
By taking $k(J^\prime )$ sufficiently larger than $k(J^\prime -1)$, one can suppose that 
all of  $\hat m_{k(J^\prime -1)}/\hat m_{k(J^\prime )}$, $Z_1$ and $Z_2$ are arbitrarily close to zero.
Thus  
there exists a sequence 
$\{ k(J^\prime )\}_{J^\prime =1}^\infty$ satisfying \eqref{eq:e0326c}.

In a similar manner, we also can get that
\begin{equation}\label{eq:e0326d}
\frac{N_{\nu  } (\omega , U_\delta (\hat p) ; \hat m_{k(J')})}{\hat m_{k(J')}}  \geq \frac{\sum _{k=k(J'-1)+1}^{k(J')} \left( {k}^2-2n_0 -n_1 \right)}{\hat m_{k(J')}} \geq \frac{1}{z_*+1}-2^{-J'}.
\end{equation}
Since $\frac{z_*}{z_*+1}+\frac{1}{z_*+1} =1$, \eqref{eq:e0326c} and \eqref{eq:e0326d} completes the proof of Theorem \ref{thm:Newhouse}  with $\lambda _1= \frac{1}{z_0+1}$, $\lambda _2= \frac{1}{z_0+2}$ and $n_1(J) =\hat m_{k(2J-1)}$, $n_2(J) =\hat m_{k(2J)}$. 


\appendix
\section{The definition of absolute continuity}\label{appendix:araujo}
In this appendix,  we compare Theorem \ref{thm:main} with Ara\'ujo's result in  \cite{Araujo2000} from the viewpoint of absolute continuity of a parametrised family of $\Ci^r$ diffeomorphisms. 
A \emph{parametrised family} $\tilde f$ of $\Ci^r$ diffeomorphisms  on a closed smooth Riemannian manifold $M$ is a differential mapping  from $B \times M$ to $M$ such that $\tilde f_t \equiv  \tilde f(t,\cdot ):M\to M$ is  a $\Ci ^r $ diffeomorphism for all $t\in B$, where $B$ is the unit ball  of a Euclidean space. 
Let  $(B^{\N _0}, \mathcal B(B)^{\N _0} , \mathrm{Leb} _{B} ^{\N _0})$ be  the product space  of a probability space $(B , \mathcal B(B), \mathrm{Leb} _{B})$, where  $\mathcal B(B)$  is  the Borel $\sigma$-field of $B$ and $\mathrm{Leb} _{B}$  is the normalised Lebesgue measure on $B$. 
For each $n\geq 1$, $\bar t=(t_0,t_1,\ldots )\in B^{\N _0}$ and $x\in M$, we define $\tilde f^{(n)}_{\bar t}(x)$  by
\[
\tilde f^{(n)}_{\bar t}(x) = \tilde f_{t_{n-1}} \circ \tilde  f_{t_{n-2}} \circ \cdots \circ \tilde  f_{t_0} (x),
\]
and let $\tilde f^{(n)}_{\bar t} =\mathrm{id} _M$ for all $\bar t\in B^{\N _0}$.
Let  $\mathrm{Leb} _M$   be the normalised Lebesgue measure on $M$.
The following  condition is from \cite[Theorem 1]{Araujo2000}. 
\begin{dfn}
 Let $\tilde f:B\times M\to M$ be 
a parametrised family of $\Ci^r$ diffeomorphisms. 
 We say that $\tilde f$ is 
\emph{absolutely continuous}  if  there exists an integer $N\geq 1$ and a real number $\xi >0$ such that for all $n\geq N$ and $x\in M$,
\begin{align}\label{eq:uniform}
&\text{$\{\tilde f_{\bar t} ^{(n)}(x) \mid \bar t \in B^{\N _0}   \}$ contains   the ball with radius $\xi $ centred at $f_0^n(x)$}, \\
&\label{eq:uniform2} \text{$\left(\tilde f_{(\cdot )} ^{(n)}(x)\right)_* \mathrm{Leb} _B^{\mathbb N_0}$ is absolutely continuous with respect to $\mathrm{Leb} _M$,} 
\end{align}
where $0$ is the centre of $B$ and 
$(\tilde f_{(\cdot)} ^{(n)}(x))_* $ is the pushforward of measures by  $\tilde f_{(\cdot)} ^{(n)}(x):B^{\N _0} \to M$ (the measurability of $\tilde f_{(\cdot)} ^{(n)}(x)$ is ensured by \cite[Property 2.1]{Araujo2000}).
\end{dfn}
Note that  the deterministic case  (i.e.,  the case when $\tilde f_t =\tilde f_0$ for all $t \in B$)  is excluded by assuming that $\tilde f$ is absolutely continuous. 

Let $\theta :B^{\N _0} \to B^{\N _0}$ be  the \emph{one-sided shift}, i.e., a measurable mapping given by $\theta (\bar t)  =  (t_1, t_2,\ldots )$ for each $\bar t =(t_0,t_1,\ldots )$.
Given a parametrised family $\tilde f : B\times M\to M$ of $\Ci^r$ diffeomorphisms, we define a mapping $F: \N _0 \times B^{\N _0} \times M \to M$  by 
\[
F(n,\bar t ,x)= \tilde f^{(n)}_{\bar t}(x), 
\quad (n,\bar t ,x) \in \N _0 \times B^{\N _0} \times M.
\]
Then, it is straightforward to see that 
 $F$ is an NDS over $\theta$ on base space $(\Omega , \mathcal F,\mathbb P)=(B^{\N _0}, \mathcal B(B)^{\N _0} , \mathrm{Leb} _{B} ^{\N _0})$.
Notice that the $B$-valued random process $\{ \omega =(t_0,t_1,\ldots ) \mapsto t_n\} _{n\geq 0}$ on $(\Omega , \mathcal F , \mathbb P)$ is independent and identically distributed, so that  we call $F$  an \emph{i.i.d.~nonautonomous dynamical system} of $\tilde f$.
We also note that $\theta$ is measure-preserving, i.e., an i.i.d.~NDS is a random dynamical system.

 The following theorem is an immediate consequence of \cite[Theorem 1]{Araujo2000}.
\begin{thm}[Ara\'ujo]\label{thm:araujo}
Let  $F$ be an  i.i.d.~nonautonomous dynamical system of a parametrised family $\tilde f$ of $\Ci ^r$ diffeomorphisms. 
Suppose that $\tilde f$ is absolutely continuous.
Then  for $\mathbb P $-almost every $\omega \in \Omega$, the set of points with  historic behaviour at $\omega$ is a zero Lebesgue measure set.
\end{thm}

We show that our parametrised perturbation is also absolutely continuous. 
In order to avoid notational confusion, we introduce another form for the mapping given in \eqref{eq:defoffepsilon}.
Let $M$ be the circle $\Sone =\R /\Z$ and $B$ the unit disk of a Euclidean space. We define a parametrised perturbation $\tilde f : B\times M\to M$ of $\Ci ^r$ diffeomorphisms by
\begin{equation}\label{eq:n2}
\tilde f(t,x) = f_0(x) +\epsilon \kappa(t) \mod 1, \quad (t,x) \in B\times M,
\end{equation}
where $f_0$ is the $\Ci ^r$ mapping given in \eqref{eq:n1} and $\kappa :B\to [-1,1]$ is a surjective continuous function such that $\kappa _* \mathrm{Leb} _B$ is absolute continuous with respect to $\mathrm{Leb} _M$. 
\begin{prop}\label{prop:11}
Let $\tilde f$ be a parametrised family of $\Ci ^r$ diffeomorphisms given in  \eqref{eq:n2}. 
Then $\tilde f$ is absolutely continuous. 
\end{prop}
We note that, although the parametrised family $\tilde f$ is absolutely continuous, the driving system of our NDS
 in Theorem \ref{thm:main} is  completely
different  from the driving system  of i.i.d.~NDS's (i.e.,~the one-sided shift) in the sense of historic behaviour (as  in Remark \ref{rmk:1}), which may cause the difference between our and Ara\'ujo's results.
\begin{proof}[Proof of Proposition \ref{prop:11}]
We first see that \eqref{eq:uniform} holds, so fix $n\geq 1$ and $x\in \Sone$. 
Due to that $\tilde f^{(n)}_{\bar t}(x) =\tilde f_{t_{n-1}}(\tilde f_{\bar t}^{(n-1)} (x))$ for each $\bar t \in B^{\N _0}$,  $\{ \tilde f^{(n)}_{\bar t}(x) \mid \bar t\in B^{\N _0}\}$ contains $\{\tilde f_t(\tilde f^{n-1}_0(x)) \mid t\in B\}$. 
Furthermore, by virtue of  \eqref{eq:n2}, $\{\tilde f_t(\tilde f^{n-1}_0(x)) \mid t\in B\}$ coincides with the ball with radius $\epsilon $ centred at $\tilde f_0^{n-1}(x)$. Therefore  \eqref{eq:uniform} holds  with  $N=1$ and $\xi =\epsilon $.

Arguing by induction, we first see that \eqref{eq:uniform2} holds for $n=1$. 
For any $x\in \Sone$ and Borel set $A \subset \Sone$, if we let $\frac{A-f_0(x)}{\epsilon} =\{ y\in \mathbb R \mid f_0(x) +\epsilon y \in A\}$, 
then 
\[
\left(\tilde f^{(1)}_{(\cdot)} (x)\right)_* \mathrm{Leb} _B ^{\N _0} (A) =\mathrm{Leb}_{B} \left( \left\{ t \in B \mid f_0(x)+\epsilon \kappa (t) \in \tilde A\right\} \right) 
\]
coincides with $\kappa _*\mathrm{Leb} _B(\frac{A-f_0(x)}{\epsilon})$. If $\mathrm{Leb} _{\mathbb R} (A)=0$, then obviously $\mathrm{Leb} _{\mathbb R} (\frac{A-f_0(x)}{\epsilon})=0$, so we get $\kappa _*\mathrm{Leb} _B(\frac{A-f_0(x)}{\epsilon}) =0$ due to the absolute continuity of $\kappa _*\mathrm{Leb} _B$. 
That is, \eqref{eq:uniform2} holds with $n=1$. 

Suppose that \eqref{eq:uniform2} holds for $n=k$. 
For any $x\in \Sone$ and Borel set $A\subset \Sone$, $(\tilde f^{(k+1)}_{(\cdot )}(x))_* \mathrm{Leb} _B ^{\N _0} (A)$ coincides with 
\begin{equation}\label{eq:rrrr3}
\int \mathrm{Leb} _B^{\N _0} \left(\left\{ \bar t  \in B^{\N _0} \mid \tilde f^{(k)} _{\bar t} (x) \in \tilde f_{t}  ^{-1} (A)\right\}\right)d\mathrm{Leb} _B (t).
\end{equation}
On the other hand, if $\mathrm{Leb} _{\mathbb R} (A)=0$, then $\mathrm{Leb} _{\mathbb R} (f_t^{-1} (A))=0$ for any $t\in B$ since $\mathrm{Leb} _{\Sone}(\tilde f_t ^{-1} (A)) \leq \sup _{x}\vert D\tilde f_t ^{-1}(x)\vert \mathrm{Leb} _{\Sone} (A)$. 
Hence, by the inductive step, we get $ \mathrm{Leb} _B^{\mathbb N_0}(\{ \bar t  \in B^{\N _0} \mid \tilde f^{(k)} _{\bar t} (x) \in \tilde f_{t}  ^{-1} (A) \} )=0$ for each $t\in B$, and  \eqref{eq:uniform2} with $n=k+1$ follows from \eqref{eq:rrrr3}. 
\end{proof}

\section{Proof of Proposition \ref{prop:4}}\label{appendix:proof}

We shall first  find an essentially bounded  mapping $Y : \Omega \to I_0$, which is invariant under $f$, i.e., $f_\omega (Y (\omega)) = Y (\theta \omega)$ $\mathbb P$-almost surely, under the identification of $I_0$ with $[\frac{1}{4} ,\frac{3}{4} ]$ by $\pi _M$. 
Let $L^\infty (\Omega , I_0)$ be the space of measurable mappings $b \in L^\infty (\Omega ,\mathbb R)$ whose essential supremum norm $\Vert b\Vert _{L^\infty}$ is in $I_0$. 
For each $b \in L^\infty (\Omega ,I_0)$, we define a mapping $\mathcal G (b) :\Omega \to I_0$ by
\[
 \mathcal G (b )  (\omega) = f_ {\theta ^{-1}\omega }\left( b (\theta ^{-1} \omega)\right) , \quad \omega \in \Omega.
\]
Then, it is easy to see that $\mathcal G (b)$  is in  $L^\infty(\Omega , I_0)$: Note that $\mathcal G(b)$ is the composition of two measurable mappings $ \omega \mapsto  f_ {\omega }\left( b ( \omega )\right) = f_0\circ b(\omega ) + \epsilon \kappa (\omega )$ 
and $\theta ^{-1}$. 
 (The transformation $\mathcal G :L^\infty  (\Omega ,I_0) \to L^\infty  (\Omega ,I_0)$ is called the \emph{graph transformation} induced by $f$.)
Furthermore, by virtue of \eqref{eq:n1} and \eqref{eq:defoffepsilon2}, 
 we have 
\begin{align*}
\left\Vert \mathcal G (b_1) -\mathcal G (b _2) \right\Vert _{L^\infty}
&= \esssup _{\omega\in\Omega}  \left\vert f_ {\theta ^{-1}\omega }\left( b_1 (\theta ^{-1} \omega)\right) -f_ {\theta ^{-1}\omega }\left( b _2(\theta ^{-1} \omega)\right) \right\vert \\
&= \frac{1}{2} \esssup _{\omega\in\Omega} \left\vert b _1(\theta ^{-1} \omega) - b _2(\theta ^{-1} \omega)\right\vert =
\frac{1}{2}  \Vert b _1 -b _2\Vert _{L^\infty}, 
\end{align*}
for all $ b _1, b _2 \in L^\infty  (\Omega ,I_0)$, 
i.e.,  $\mathcal G$ is a contraction mapping on the complete metric space $L^\infty (\Omega , I_0)$.  Therefore,  there exists a unique fixed point $Y$ of $\mathcal G$. By construction, $Y:\Omega \to I_0$ is an $f$-invariant essentially bounded mapping.

Fix  a continuous mapping $\varphi :M \to \R$. Since $M$ is compact, $\varphi$ is uniformly continuous. 
On the other hand, due to the invariance of $Y$ together with \eqref{eq:n1} and \eqref{eq:defoffepsilon}, we have
\begin{align}\label{eq:x2}
\left\vert 
 f^{(n)}_\omega (x) - 
  Y(\theta ^n\omega )        \right\vert & = 
    \left\vert f^{(n)}_\omega (x) -  f^{(n)}_\omega (Y(\omega ))        \right\vert \\
&= 
 \frac{1}{2^n} \left\vert x -Y(\omega )        \right\vert  \to 0 \notag
\end{align}
as $n$ goes to infinity for $\mathbb P$-almost every  $\omega$ and all $x$ in  $I_0$.
Therefore, a straightforward calculation shows that 
\[
\left\vert 
\frac{1}{n} \sum _{j=0}^{n-1}
 \varphi (f^{(j)}_\omega (x) )- 
 \frac{1}{n} \sum _{j=0}^{n-1}
 \varphi( Y(\theta ^n\omega )  )      \right\vert \to 0
\]
as $n$ goes to infinity for $\mathbb P$-almost every  $\omega$ and all $x$ in  $I_0$.
Since $ \varphi \circ Y$ is an integrable function, we get the conclusion from the fact that $\theta$ is non-historic.

\section*{Acknowledgments}
This work was partially supported by JSPS KAKENHI
Grant Numbers  \color{black} 26400093 \color{black} and 17K05283.

\end{document}